\documentclass[11pt]{amsart}
\usepackage{kotex}
\usepackage[utf8]{inputenc}
\usepackage[english]{babel}
\usepackage{tikz}
\usepackage{pst-node}
\usepackage{tikz-cd}
\usepackage{amsmath,amsfonts,amsthm}
\usepackage{mathtools}

\theoremstyle{plain}
\newtheorem{theorem}{Theorem}[section]
\newtheorem*{theorem*}{Theorem}
\newtheorem{corollary}[theorem]{Corollary}
\newtheorem*{corollary*}{Corollary}
\newtheorem{lemma}[theorem]{Lemma}
\theoremstyle{definition}

\theoremstyle{remark}

\usepackage[sort, compress]{cite}
\usepackage{url}
\usepackage{hyperref}
\newcommand{\aut}{\operatorname{Aut}}

\begin{document}

\title[An infinite family of braid group representations]
{An infinite family of braid group representations in automorphism groups of free groups}

\author{Wonjun Chang}
\address{Wonjun Chang \\ Department of Mathematics \\ Inha University \\ Incheon, 22212, Republic of Korea}
\email{wonjun5@gmail.com}

\author{Byung Chun Kim}
\address{Byung Chun Kim \\ National Institute for Mathematical Sciences \\ Daejeon, 34047, Republic of Korea}
\email{wizardbc@gmail.com}

\author{Yongjin Song}
\address{Yongjin Song \\ Department of Mathematics \\ Inha University \\ Incheon, 22212, Republic of Korea}
\email{yjsong@inha.ac.kr}

\subjclass{57M12, 57M50}
\keywords{Braid group embedding, branched covering, mapping class group, automorphism of free group}

\begin{abstract}
  The $d$-fold ($d \geq 3$) branched coverings on a disk give an infinite family of nongeometric embeddings of braid groups into mapping class groups.
  We, in this paper, give new explicit expressions of these braid group representations into automorphism groups of free groups in terms of the actions on the generators of free groups.
  We also give a systematic way of constructing and expressing these braid group representations in terms of a new gadget, called covering groupoid.
  We prove that each generator $\widetilde{\beta}_i$ of braid group inside mapping class group induced by $d$-fold covering is the product of $d-1$ Dehn twists on the surface.
\end{abstract}

\maketitle

\section{Introduction}

  The classical embedding $\psi : B_n \rightarrow \Gamma_{g,1}$ of braid groups in mapping class groups was interpreted by Segal and Tillmann \cite{SegalTillmann2008} as being induced by a 2-fold (branched) covering over a disk.
  The question on the existence of such an embedding induced by 3-fold covering was raised by Tillmann about ten years ago, and the answer was given in \cite{KimSong2018}.
  Soon after the paper \cite{KimSong2018} was posted at math arXiv in February 2018,
  similar results to this were posted at arXiv by Ghaswala-McLeay \cite{Ghaswala2018}
  and Callegaro-Salvetti \cite{Salvetti2018} for arbitrary $d$-fold ($d \geq 3$) covering.
  It is surprising that the ten-year old problem was resolved independently almost at the same time by three different groups.
  The paper \cite{Ghaswala2018} laid emphasis on the development and application of the Birman-Hilden theory
  and the paper \cite{Salvetti2018} did on the computation of integral homology of braid group with coefficient in the first homology of the surface with this new action,
  whereas the paper \cite{KimSong2018} did on the proof of the homology triviality of the embedding.

  The classical embedding $\psi : B_n \rightarrow \Gamma_{g,b}$ maps each generator of braid group to a Dehn twist.
  Such an embedding is called a geometric embedding.
  Since the question whether there exists a nongeometric embedding was raised by Wajnryb (\cite{Wajnryb2006}),
  only a few particular examples have been found (\cite{JeongSong2013,Szepietowski2010,KimSong2018}).
  The embedding induced by $d$-fold covering ($d \geq 3$) over a disk, which turned out to be a product of $d-1$ Dehn twists, gives us an infinite family of nongeometric embeddings.
  On the other hand, since the mapping class group $\Gamma_{g,b}$ ($b \geq 1$) may be regarded as a subgroup of the automorphism group of the fundamental group of the surface which is a free group,
  we also could get a family of {\it faithful} representations of braid groups into the automorphim groups of free groups.
  However, it is not easy to find an explicit expression of these representations in terms of the action on the generators of free group.
  Such new expressions of braids as automorphisms of free group may give us (theoretical) new group invariants of links (\cite{Wada1992,CrispParis2005} )
  and add some new data in the theory of braid cryptography (\cite{Ko2000} ).

  In this paper we give three results.
  First, we give new explicit expressions of those braid group representations into automorphism groups of free groups in terms of the actions on the generators of free groups (Corollary~\ref{cor:yij}).

  \begin{corollary*}{\bf \ref{cor:yij}}
    Let $\phi_d : B_n \rightarrow \aut F_{(d-1)(n-1)}$, $\beta_i \mapsto \widetilde{\beta_i}$ be the (faithful) representation induced by the $d$-fold covering over a disk with $n$ branch points.
    Then $\widetilde{\beta_{i}}$ acts on the $(d-1)(n-1)$ generators $\{ x_{1,1}, \ldots, x_{n-1,d-1} \}$ of the free group as follows:
    $$
    \widetilde\beta_i:\left\{
    \begin{array}{rcl}
      x_{i-1, j} & \mapsto &  (x_{i, j+1})^{y_{i-1,j}} \cdot x_{i-1,j}\\
      \\
      x_{i, j} & \mapsto & (x_{i,j+1}^{-1})^{y_{i,j+1}^{-1} \cdot\, x_{i,1}}\\
      \\
      x_{i+1, j} & \mapsto & (x_{i+1,j})^{y_{i,j}} \cdot x_{i,j}

    \end{array}\right.
    $$
  \end{corollary*}

  The generators of the free group have been determined from the covering groupoid.
  The faithfullness of these representations is shown by the Birman-Hilden theory.

  Second, we give a more systematic way of constructing and expressing the braid group embedding induced by arbitrary $d$-fold covering.
  For this we introduce a new gadget, called {\it covering groupoid} which is a directed graph and also regarded as a subcategory of the fundamental groupoid of the surface.

  Third, we prove that each generator $\widetilde{\beta}_i$ of braid group inside mapping class group induced by $d$-fold covering is the product of $d-1$ Dehn twists ({\it not} the inverse of Dehn twists) along closed curves which form a chain:

  \begin{theorem*}{\bf \ref{thm:Dehn}}
    The lift $\widetilde\beta_i$ equals $D_{x_{i, 2}}\cdot D_{x_{i,3}}\cdot \cdots \cdot D_{x_{i,d}}$ as an element of mapping class group.
  \end{theorem*}

\section{Mapping class groups and branched covering over a disk}

  Let $S_{g,b}$ be an oriented surface with genus $g$ and having $b$ boundary components.
  Let $\Gamma_{g,b}$ be the mapping class group of $S_{g,b}$,
  i.e. the group of isotopy classes of orientation-preserving self-homeomorphisms of $S_{g,b}$, fixing the boundary pointwise.
  It is generated by a finite number of Dehn twists.

  Let $\Gamma_{g,b}^n$ denote the mapping class group of the surface $S_{g,b}^n$ with $n$ marked points.
  Let $\Gamma_{g,b}^{(n)}$ denote the mapping class group of self-homeomorphisms which may permutate the $n$ marked points.
  Note that the braid group $B_n$ may be defined to be $\Gamma_{0,1}^{(n)}$
  and the standard generators $\{ \beta_1, \ldots \beta_{n-1} \}$ of $B_n$ are given as half Dehn twists.
  The generators $\beta_i$ satisfy the {\it braid relations}
  $\beta_i \beta_{i+1} \beta_i = \beta_{i+1}\beta_i \beta_{i+1}$ and $\beta_i\beta_j = \beta_j\beta_i \textrm{ if } |i - j| \geq 2$.
  The generator $\beta_i$ is a $180^\circ$ rotation interchanging the marked points $i$ and $i+1$ of the disk as shown in Figure~\ref{fig:halfDehn}.

  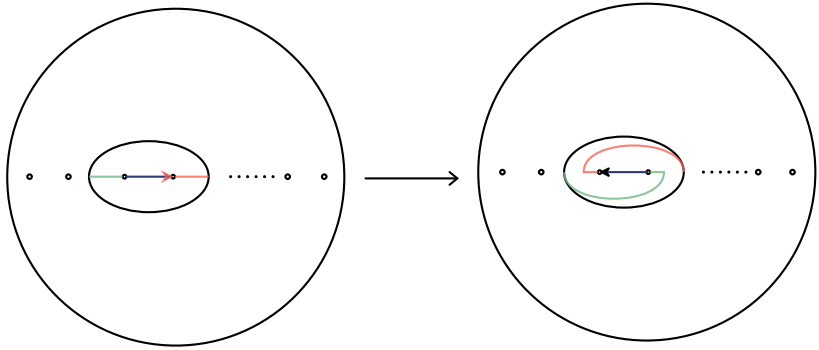
\begin{figure}[ht]
    \centering
    \definecolor{c2f3b73}{RGB}{47,59,115}
\definecolor{cf56356}{RGB}{245,99,86}
\definecolor{c74b587}{RGB}{116,181,135}

\begin{tikzpicture}[y=0.80pt, x=0.80pt, yscale=-1.000000, xscale=1.000000, inner sep=0pt, outer sep=0pt]
\path[draw=black,line join=round,line cap=round,line width=0.800pt]
  (224.6400,563.7600) .. controls (224.6400,554.4810) and (237.3090,546.9580) ..
  (252.9380,546.9580) .. controls (268.5660,546.9580) and (281.2350,554.4810) ..
  (281.2350,563.7600) .. controls (281.2350,573.0390) and (268.5660,580.5620) ..
  (252.9380,580.5620) .. controls (237.3090,580.5620) and (224.6400,573.0390) ..
  (224.6400,563.7600) -- cycle;
\path[draw=black,line join=round,line cap=round,line width=0.800pt]
  (240.5570,563.7600) .. controls (240.5570,563.2720) and (240.9530,562.8760) ..
  (241.4420,562.8760) .. controls (241.9300,562.8760) and (242.3260,563.2720) ..
  (242.3260,563.7600) .. controls (242.3260,564.2490) and (241.9300,564.6440) ..
  (241.4420,564.6440) .. controls (240.9530,564.6440) and (240.5570,564.2490) ..
  (240.5570,563.7600) -- cycle;
\path[draw=black,line join=round,line cap=round,line width=0.800pt]
  (263.5490,563.7600) .. controls (263.5490,563.2720) and (263.9450,562.8760) ..
  (264.4340,562.8760) .. controls (264.9220,562.8760) and (265.3180,563.2720) ..
  (265.3180,563.7600) .. controls (265.3180,564.2490) and (264.9220,564.6440) ..
  (264.4340,564.6440) .. controls (263.9450,564.6440) and (263.5490,564.2490) ..
  (263.5490,563.7600) -- cycle;
  \path[draw=c2f3b73,line join=round,line cap=round,line width=0.800pt]
    (261.8160,563.7600) .. controls (259.8240,563.7600) and (242.3260,563.7600) ..
    (242.3260,563.7600);
  \path[draw=cf56356,fill=c2f3b73,draw opacity=0.799,line width=0.800pt]
    (261.3820,563.7600) -- (260.0830,565.4930) -- (263.5480,563.7600) --
    (260.0830,562.0270) -- (261.3820,563.7600) -- cycle;
\path[draw=c74b587,draw opacity=0.798,line width=0.800pt] (224.6400,563.7600) --
  (240.5570,563.7600);
\path[draw=cf56356,draw opacity=0.799,line width=0.800pt] (265.3180,563.7600) --
  (281.2350,563.7600);
\path[draw=black,line join=round,line cap=round,line width=0.800pt]
  (449.2800,561.6000) .. controls (449.2800,552.3210) and (461.9490,544.7980) ..
  (477.5780,544.7980) .. controls (493.2060,544.7980) and (505.8750,552.3210) ..
  (505.8750,561.6000) .. controls (505.8750,570.8790) and (493.2060,578.4020) ..
  (477.5780,578.4020) .. controls (461.9490,578.4020) and (449.2800,570.8790) ..
  (449.2800,561.6000) -- cycle;
\path[draw=black,line join=round,line cap=round,line width=0.800pt]
  (465.1970,561.6000) .. controls (465.1970,561.1120) and (465.5930,560.7160) ..
  (466.0820,560.7160) .. controls (466.5700,560.7160) and (466.9660,561.1120) ..
  (466.9660,561.6000) .. controls (466.9660,562.0890) and (466.5700,562.4840) ..
  (466.0820,562.4840) .. controls (465.5930,562.4840) and (465.1970,562.0890) ..
  (465.1970,561.6000) -- cycle;
\path[draw=black,line join=round,line cap=round,line width=0.800pt]
  (488.1890,561.6000) .. controls (488.1890,561.1120) and (488.5850,560.7160) ..
  (489.0740,560.7160) .. controls (489.5620,560.7160) and (489.9580,561.1120) ..
  (489.9580,561.6000) .. controls (489.9580,562.0890) and (489.5620,562.4840) ..
  (489.0740,562.4840) .. controls (488.5850,562.4840) and (488.1890,562.0890) ..
  (488.1890,561.6000) -- cycle;
  \path[draw=c2f3b73,line join=round,line cap=round,line width=0.800pt]
    (468.7000,561.6000) .. controls (470.6920,561.6000) and (488.1890,561.6000) ..
    (488.1890,561.6000);
  \path[draw=black,fill=c2f3b73,line join=round,line cap=round,line width=0.800pt]
    (469.1330,561.6000) -- (470.4330,559.8670) -- (466.9670,561.6000) --
    (470.4330,563.3330) -- (469.1330,561.6000) -- cycle;
\path[draw=c74b587,draw opacity=0.798,line width=0.800pt] (449.2800,561.6000) ..
  controls (449.2800,578.4020) and (496.5250,578.4020) .. (496.5250,561.6000) --
  (489.9580,561.6000);
\path[draw=cf56356,draw opacity=0.799,line width=0.800pt] (465.1970,561.6000) --
  (458.5260,561.6000) .. controls (458.5260,544.7980) and (505.8750,544.7980) ..
  (505.8750,561.6000);
\path[draw=black,line join=round,line cap=round,line width=0.800pt]
  (185.9900,563.9900) .. controls (185.9900,608.0020) and (221.6680,643.6800) ..
  (265.6800,643.6800) .. controls (309.6920,643.6800) and (345.3700,608.0020) ..
  (345.3700,563.9900) .. controls (345.3700,519.9790) and (309.6920,484.3000) ..
  (265.6800,484.3000) .. controls (221.6680,484.3000) and (185.9900,519.9790) ..
  (185.9900,563.9900) -- cycle;
\path[draw=black,line join=round,line cap=round,line width=0.800pt]
  (334.8000,563.7600) .. controls (334.8000,563.1630) and (335.2840,562.6800) ..
  (335.8800,562.6800) .. controls (336.4760,562.6800) and (336.9600,563.1630) ..
  (336.9600,563.7600) .. controls (336.9600,564.3560) and (336.4760,564.8400) ..
  (335.8800,564.8400) .. controls (335.2840,564.8400) and (334.8000,564.3560) ..
  (334.8000,563.7600) -- cycle;
\path[draw=black,line join=round,line cap=round,line width=0.800pt]
  (195.4800,563.7600) .. controls (195.4800,563.1630) and (195.9640,562.6800) ..
  (196.5600,562.6800) .. controls (197.1560,562.6800) and (197.6400,563.1630) ..
  (197.6400,563.7600) .. controls (197.6400,564.3560) and (197.1560,564.8400) ..
  (196.5600,564.8400) .. controls (195.9640,564.8400) and (195.4800,564.3560) ..
  (195.4800,563.7600) -- cycle;
\path[draw=black,line join=round,line cap=round,line width=0.800pt]
  (317.5200,563.7600) .. controls (317.5200,563.1630) and (318.0040,562.6800) ..
  (318.6000,562.6800) .. controls (319.1960,562.6800) and (319.6800,563.1630) ..
  (319.6800,563.7600) .. controls (319.6800,564.3560) and (319.1960,564.8400) ..
  (318.6000,564.8400) .. controls (318.0040,564.8400) and (317.5200,564.3560) ..
  (317.5200,563.7600) -- cycle;
\path[draw=black,line join=round,line cap=round,line width=0.800pt]
  (213.8400,563.7600) .. controls (213.8400,563.1630) and (214.3240,562.6800) ..
  (214.9200,562.6800) .. controls (215.5160,562.6800) and (216.0000,563.1630) ..
  (216.0000,563.7600) .. controls (216.0000,564.3560) and (215.5160,564.8400) ..
  (214.9200,564.8400) .. controls (214.3240,564.8400) and (213.8400,564.3560) ..
  (213.8400,563.7600) -- cycle;
  \path[draw=black,line join=round,line cap=round,line width=0.800pt]
    (398.3100,564.5800) .. controls (395.0840,564.5800) and (355.3760,564.5800) ..
    (355.3760,564.5800);
  \path[fill=black] (395.4600,567.9670) -- (399.1270,564.9670) --
    (399.6000,564.5800) -- (399.1270,564.1930) -- (395.4600,561.1930) .. controls
    (395.2470,561.0180) and (394.9320,561.0500) .. (394.7570,561.2630) .. controls
    (394.5820,561.4770) and (394.6130,561.7920) .. (394.8270,561.9670) --
    (398.4940,564.9670) -- (398.4940,564.1930) -- (394.8270,567.1930) .. controls
    (394.6130,567.3680) and (394.5820,567.6830) .. (394.7570,567.8970) .. controls
    (394.9320,568.1100) and (395.2470,568.1420) .. (395.4600,567.9670) -- cycle;
\path[draw=black,line join=round,line cap=round,line width=0.800pt]
  (408.7000,561.6000) .. controls (408.7000,605.6120) and (444.3780,641.2900) ..
  (488.3900,641.2900) .. controls (532.4020,641.2900) and (568.0800,605.6120) ..
  (568.0800,561.6000) .. controls (568.0800,517.5890) and (532.4020,481.9100) ..
  (488.3900,481.9100) .. controls (444.3780,481.9100) and (408.7000,517.5890) ..
  (408.7000,561.6000) -- cycle;
\path[draw=black,line join=round,line cap=round,line width=0.800pt]
  (556.2000,561.6000) .. controls (556.2000,561.0030) and (556.6830,560.5200) ..
  (557.2800,560.5200) .. controls (557.8760,560.5200) and (558.3600,561.0030) ..
  (558.3600,561.6000) .. controls (558.3600,562.1960) and (557.8760,562.6800) ..
  (557.2800,562.6800) .. controls (556.6830,562.6800) and (556.2000,562.1960) ..
  (556.2000,561.6000) -- cycle;
\path[draw=black,line join=round,line cap=round,line width=0.800pt]
  (540.0000,561.6000) .. controls (540.0000,561.0030) and (540.4840,560.5200) ..
  (541.0800,560.5200) .. controls (541.6760,560.5200) and (542.1600,561.0030) ..
  (542.1600,561.6000) .. controls (542.1600,562.1960) and (541.6760,562.6800) ..
  (541.0800,562.6800) .. controls (540.4840,562.6800) and (540.0000,562.1960) ..
  (540.0000,561.6000) -- cycle;
\path[draw=black,dash pattern=on 0.00pt off 3.20pt,line join=round,line
  cap=round,line width=1.200pt] (291.6000,563.7600) -- (315.3600,563.7600);
\path[draw=black,dash pattern=on 0.00pt off 3.20pt,line join=round,line
  cap=round,line width=1.200pt] (515.1600,561.6000) -- (538.9200,561.6000);
\path[draw=black,line join=round,line cap=round,line width=0.800pt]
  (419.0400,561.6000) .. controls (419.0400,561.0030) and (419.5240,560.5200) ..
  (420.1200,560.5200) .. controls (420.7160,560.5200) and (421.2000,561.0030) ..
  (421.2000,561.6000) .. controls (421.2000,562.1960) and (420.7160,562.6800) ..
  (420.1200,562.6800) .. controls (419.5240,562.6800) and (419.0400,562.1960) ..
  (419.0400,561.6000) -- cycle;
\path[draw=black,line join=round,line cap=round,line width=0.800pt]
  (437.4000,561.6000) .. controls (437.4000,561.0030) and (437.8830,560.5200) ..
  (438.4800,560.5200) .. controls (439.0760,560.5200) and (439.5600,561.0030) ..
  (439.5600,561.6000) .. controls (439.5600,562.1960) and (439.0760,562.6800) ..
  (438.4800,562.6800) .. controls (437.8830,562.6800) and (437.4000,562.1960) ..
  (437.4000,561.6000) -- cycle;

\end{tikzpicture}
    \caption{A half Dehn twist acting on a disk}
    \label{fig:halfDehn}
  \end{figure}

  Mapping class groups are generated by full Dehn twists. Let $\alpha$ be a simple closed curve in an orientable surface $S$. Let $D_{\alpha}$ denote the Dehn twist along $\alpha$.
  For recalling the orientation of the twist of $D_{\alpha}$, see Figure~\ref{fig:Dehn}.

  \begin{figure}[ht]
    \centering
    \definecolor{cff0025}{RGB}{255,0,37}

\begin{tikzpicture}[y=0.80pt, x=0.80pt, yscale=-1.000000, xscale=1.000000, inner sep=0pt, outer sep=0pt]
  \path[draw=black,line width=0.480pt] (225.6740,169.6330);
  \path[draw=black,line width=0.480pt] (271.3200,203.0400) .. controls
    (271.3200,184.5900) and (275.0590,169.6330) .. (279.6720,169.6330);
  \begin{scope}[shift={(-3.0,0)}]
      \path[draw=black,line width=0.480pt] (274.6330,211.5230) .. controls
        (274.4290,208.6310) and (274.3200,205.5790) .. (274.3200,202.4250);
      \path[fill=black] (272.4810,209.9670) -- (274.4380,212.0280) --
        (274.6900,212.2940) -- (274.9010,211.9940) -- (276.5340,209.6680) .. controls
        (276.6290,209.5320) and (276.5960,209.3450) .. (276.4610,209.2500) .. controls
        (276.3250,209.1540) and (276.1380,209.1870) .. (276.0430,209.3230) --
        (274.4100,211.6490) -- (274.8730,211.6150) -- (272.9160,209.5540) .. controls
        (272.8020,209.4340) and (272.6120,209.4290) .. (272.4920,209.5430) .. controls
        (272.3720,209.6570) and (272.3670,209.8470) .. (272.4810,209.9670) -- cycle;
  \end{scope}
  \path[draw=black,line width=0.480pt] (279.6310,236.4090) .. controls
    (275.8770,236.4090) and (272.7020,226.5040) .. (271.6490,212.8710);
  \path[draw=black,line width=0.480pt] (279.9600,169.6330) .. controls
    (284.5730,169.6330) and (288.3120,184.5900) .. (288.3120,203.0400) .. controls
    (288.3120,221.4900) and (284.5730,236.4470) .. (279.9600,236.4470);
  \path[draw=black,line width=0.480pt] (309.1910,169.6330);
  \path[draw=black,line width=0.480pt] (225.6740,169.6330) -- (280.6150,169.6330);
  \path[draw=black,line width=0.480pt] (225.6740,236.4470) -- (280.6980,236.4470);
  \path[draw=cff0025,line width=0.480pt] (253.7520,236.4470) .. controls
    (249.1390,236.4470) and (245.4000,221.4900) .. (245.4000,203.0400) .. controls
    (245.4000,184.5900) and (249.1390,169.6330) .. (253.7520,169.6330);
  \path[draw=cff0025,dash pattern=on 1.60pt,line width=0.480pt]
    (254.0400,169.6330) .. controls (258.6530,169.6330) and (262.3920,184.5900) ..
    (262.3920,203.0400) .. controls (262.3920,221.4900) and (258.6530,236.4470) ..
    (254.0400,236.4470);
  \path[draw=black,line width=0.480pt] (225.6730,236.4470) .. controls
    (221.0610,236.4470) and (217.3220,221.4900) .. (217.3220,203.0400) .. controls
    (217.3220,184.5900) and (221.0610,169.6330) .. (225.6730,169.6330);
  \path[draw=black,dash pattern=on 1.60pt,line width=0.480pt] (225.6730,169.6330)
    .. controls (230.2860,169.6330) and (234.0250,184.5900) .. (234.0250,203.0400)
    .. controls (234.0250,221.4900) and (230.2860,236.4470) ..
    (225.6730,236.4470);
  \begin{scope}[shift={(-3.0,0)}]
      \path[draw=black,line width=0.480pt] (270.0000,156.2940) .. controls
        (270.0000,158.0470) and (270.0000,178.4570) .. (270.0000,178.4570);
      \path[fill=black] (272.0320,158.0040) -- (270.2320,155.8040) --
        (270.0000,155.5210) -- (269.7680,155.8040) -- (267.9680,158.0040) .. controls
        (267.8630,158.1330) and (267.8820,158.3220) .. (268.0100,158.4260) .. controls
        (268.1380,158.5310) and (268.3270,158.5120) .. (268.4320,158.3840) --
        (270.2320,156.1840) -- (269.7680,156.1840) -- (271.5680,158.3840) .. controls
        (271.6730,158.5120) and (271.8620,158.5310) .. (271.9900,158.4260) .. controls
        (272.1180,158.3220) and (272.1370,158.1320) .. (272.0320,158.0040) -- cycle;
  \end{scope}
  \path[fill=black] (292.3600,187.9200) node[above right] (text52) {$360^{\circ}
    \textrm{rotation}$};
  \begin{scope}[shift={(-3.0,0)}]
    \path[fill=black] (194.4000,220.3200) -- (213.8400,209.5200);
      \path[draw=black,line width=0.480pt] (213.1630,209.8960) .. controls
        (211.6640,210.7290) and (194.4000,220.3200) .. (194.4000,220.3200);
      \path[fill=black] (212.6550,212.5030) -- (213.7040,209.8610) --
        (213.8390,209.5200) -- (213.4790,209.4550) -- (210.6810,208.9500) .. controls
        (210.5180,208.9210) and (210.3620,209.0290) .. (210.3330,209.1920) .. controls
        (210.3030,209.3550) and (210.4120,209.5110) .. (210.5750,209.5400) --
        (213.3720,210.0460) -- (213.1460,209.6400) -- (212.0970,212.2820) .. controls
        (212.0360,212.4360) and (212.1120,212.6100) .. (212.2660,212.6710) .. controls
        (212.4200,212.7320) and (212.5940,212.6570) .. (212.6550,212.5030) -- cycle;
  \end{scope}
  \path[fill=black] (178.4400,235.4000) node[above right] (text66)
    {$\textrm{Fix}$};
  \path[fill=black] (236.7600,140.4000) node[above right] (text70)
    {$\textrm{Outward normal}$};
  \path[fill=black] (236.7600,198.3600) node[above right] (text74) {$\alpha$};
  \path[xscale=1.042,yscale=0.959,fill=black,line width=0.235pt]
    (202.8617,171.6118) node[above right] (text1045) {$N$};

\end{tikzpicture}
    \caption{Dehn twist along the curve $\alpha$}
    \label{fig:Dehn}
  \end{figure}
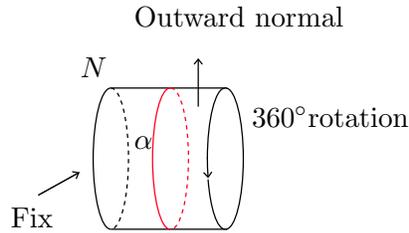

  \subsection{Cyclic $d$-fold branched cover over a disk}

    Let $\pi: S_g\rightarrow S^{2} \cong \mathbb C \cup \{\infty\}$ be a cyclic branched cover
    (or superelliptic curve) modelled by $$y^d=(z-1)(z-2)\cdots(z-n)$$ on a sphere.
    Then there is only one ramification point over each  $i$ of index $d$ ($i=1,2,\ldots,n$).
    And there are $\gcd(d,n)$ ramification points over $\infty$ of index $\frac{d}{\gcd(d,n)}$ .
    Riemann-Hurwitz formula gives $$\chi(S_g) = d-dn-n+\gcd(d,n).$$
    Since $\chi(S_g) = 2-2g$, $$g=\frac{dn-n-d-\gcd(d,n)}{2}+1.$$
    A {\it $d$-fold branched cover with $n$ branch points} over a disk
    $$S_{0,1}^{n}=\left\{z : \left|z-\frac{n+1}{2}\right|\leq\frac{n+1}{2}\right\}$$
    with $n$ marked points $\{1,2,\ldots,n\}$
    is the restriction $\pi|_{E}:E\rightarrow S_{0,1}^{n}$ where $E=\pi^{-1}\left(S_{0,1}^{n}\right)$.
    The boundary of the disk $\partial S_{0,1}^{n}$ can be thought of as a loop around $\infty$ on $S^{2}$.
    So $E$ has the same number of boundary components as the number of the ramification points over $\infty$.
    Therefore $E\cong S_{g,b}^{n}$ where $g=\frac{dn-n-d-\gcd(d,n)}{2}+1$ and $b=\gcd(d,n)$.
    (Tables~\ref{tab:4fold} and \ref{tab:5fold})


    \begin{table}[h]
      \centering
      \begin{tabular}{c  c  c  c  c  c  c  c  c}
      \hline
      $n$ & 1 & 2 & 3 & 4 & 5 & 6 & 7 & 8 \\ [0.5ex]
      \hline
      $b$ & 1 & 2 & 1 & 4 & 1 & 2 & 1 & 4 \\
      $g$ & 0 & 1 & 3 & 3 & 6 & 7 & 9 & 9 \\
      \hline
      \end{tabular}
      \caption{4-fold branched covering space with $n$ branch points}
      \label{tab:4fold}
    \end{table}

    \begin{table}[h]
      \centering
      \begin{tabular}{c  c  c  c  c  c  c  c  c}
      \hline
      $n$ & 1 & 2 & 3 & 4 & 5 & 6 & 7 & 8 \\ [0.5ex]
      \hline
      $b$ & 1 & 1 & 1 & 1 & 5 & 1 & 1 & 1 \\
      $g$ & 0 & 2 & 4 & 6 & 6 & 10 & 12 & 14 \\
      \hline
      \end{tabular}
      \caption{5-fold branched covering space with $n$ branch points}
      \label{tab:5fold}
    \end{table}

  \subsection{Covering Groupoid}

    For a branched cover over a disk, we are going to use the language of groupoid.

    Let $\pi:E\rightarrow D$ be a $d$-fold branched covering with $n$ branch points $\{ 1, \ldots, n\}$.
    This disk $D$ may be interpreted as a groupoid $D$ (abuse of notation) which is a directed graph such that the set of vertices is $\{0,1,\ldots,n,n+1\}$ and the edges are generated by $\left\{i\xrightarrow{e_i} i+1\right\}_{i=0,1,\ldots,n}$.
    The points 0 and $n$ are thought to lie at the boundary of the disk.

    The groupoid $E$ (again, abuse of notation) corresponding to $\pi$, called {\it covering groupoid}, is defined as follows:
    $$
    E = \left\{
    \begin{array}{rcl}
    \textrm{vertices} & : & \left\{1, 2,\ldots,n\right\}\quad\bigcup\quad \left\{0_{(j)}, (n+1)_{(j)}\right\}_{j=1,2,\ldots,d}\\
    \textrm{edges} & : & \textrm{generated by }
    \left\{i\xrightarrow{e_{i, j}} i+1\right\}_{i=0,1,\ldots,n,\  j=1,2,\ldots,d}\\
    \end{array}\right.
    $$
    where $0\xrightarrow{e_{0,j}} 1$ and $n\xrightarrow{e_{n, j}} n+1$ mean
    $0_{(j)}\xrightarrow{e_{0,j}} 1$ and $n\xrightarrow{e_{n, j}} (n+1)_{(j)}$, respectively,
    and the index $j$ reads modulo $d$.

    Note that for a covering $\pi:E\rightarrow D$, the surface $E$ may be obtained as a tubular neighborhood of the graph $E$.

    In $D$, between two branch points $i$ and $i+1$ there is only one edge $e_i$,
    whereas in $E$ between $i$ and $i+1$ there are $d$ edges $e_{i,1}, e_{i,2}, \ldots, e_{i,d}$.
    And in $E$ at the point $i+1$, in-edges $e_{i,1}, \ldots, e_{i,d}$ and out-edges $e_{i+1,1}, \ldots, e_{i+1,d}$ should meet alternately (Figure~\ref{fig:pattern}).

    \begin{figure}[ht]
      \centering
      \begin{tikzpicture}[y=0.80pt, x=0.80pt, yscale=-1.000000, xscale=1.000000, inner sep=0pt, outer sep=0pt]
\path[draw=black,fill=black,opacity=0.799,line width=1.600pt]
  (134.6290,255.5760) -- (273.5350,280.0690);
\path[draw=black,fill=black,opacity=0.799,line width=1.600pt]
  (134.6290,256.0610) -- (273.5350,231.5680);
\path[draw=black,fill=black,opacity=0.799,line width=1.600pt]
  (134.7400,255.6760) -- (256.8920,185.1510);
\path[draw=black,fill=black,opacity=0.799,line width=1.600pt]
  (135.5190,255.5760) -- (257.6710,326.1010);
\path[draw=black,opacity=0.799,line width=1.600pt] (209.6230,262.7050) --
  (214.4260,269.6470) -- (207.5390,274.5270);
\path[draw=black,opacity=0.799,line width=1.600pt] (209.8970,236.7110) --
  (216.9610,241.3300) -- (212.4210,248.4460);
\path[draw=black,opacity=0.799,line width=1.600pt] (202.7650,288.7650) ..
  controls (204.9160,296.9270) and (204.9160,296.9270) .. (204.9160,296.9270) --
  (196.7780,299.1700);
\path[draw=black,opacity=0.799,line width=1.600pt] (201.2010,210.4760) --
  (209.4420,212.3030) -- (207.7070,220.5640);
\path[draw=black,dash pattern=on 0.80pt off 0.80pt,opacity=0.799,line
  width=0.800pt] (154.0080,276.5330) .. controls (147.6930,282.3500) and
  (138.8330,285.2590) .. (129.7250,283.6530) .. controls (114.2180,280.9190) and
  (103.8640,266.1320) .. (106.5990,250.6250) .. controls (109.3330,235.1190) and
  (124.1200,224.7650) .. (139.6260,227.4990) .. controls (146.8230,228.7680) and
  (152.9110,232.6340) .. (157.0940,237.9630);
\path[draw=black,fill=black,opacity=0.799,line width=1.600pt]
  (508.4480,261.4540) -- (370.4990,232.0460);
\path[draw=black,fill=black,opacity=0.799,line width=1.600pt]
  (508.4480,261.2860) -- (368.7600,280.8330);
\path[draw=black,fill=black,opacity=0.799,line width=1.600pt]
  (508.4480,262.0560) -- (383.8700,328.2000);
\path[draw=black,fill=black,opacity=0.799,line width=1.600pt]
  (508.6890,260.6350) -- (389.1170,185.8190);
\path[draw=black,opacity=0.799,line width=1.600pt] (430.4520,238.6900) --
  (435.0050,245.7970) -- (427.9490,250.4300);
\path[draw=black,opacity=0.799,line width=1.600pt] (427.2700,265.4600) --
  (434.1660,270.3270) -- (429.3760,277.2780);
\path[draw=black,opacity=0.799,line width=1.600pt] (445.5830,214.1640) ..
  controls (447.4060,222.4060) and (447.4060,222.4060) .. (447.4060,222.4060) --
  (439.1850,224.3210);
\path[draw=black,opacity=0.799,line width=1.600pt] (440.4250,304.8680) --
  (432.2540,302.7500) -- (434.2800,294.5550);
\path[draw=black,dash pattern=on 0.80pt off 0.80pt,opacity=0.799,line
  width=0.800pt] (490.1300,240.4270) .. controls (496.6480,234.8380) and
  (505.6050,232.2450) .. (514.6510,234.1730) .. controls (530.0500,237.4560) and
  (539.8730,252.6010) .. (536.5900,268.0000) .. controls (533.3070,283.4000) and
  (518.1620,293.2230) .. (502.7620,289.9400) .. controls (495.6140,288.4160) and
  (489.6680,284.3370) .. (485.6760,278.8630);
\path[cm={{2.10842,0.0,0.0,2.10842,(119.623,250.067)}},fill=black,opacity=0.799]
  (0.0000,0.0000) node[above right] (text46) {$i$};
\path[cm={{2.10842,0.0,0.0,2.10842,(511.449,259.055)}},fill=black,opacity=0.799]
  (0.0000,0.0000) node[above right] (text50) {$i+1$};
\path[cm={{1.80948,0.0,0.0,1.80948,(228.96,186.347)}},fill=black,opacity=0.799]
  (0.0000,0.0000) node[above right] (text54) {$e_{i,3}$};
\path[cm={{1.80948,0.0,0.0,1.80948,(250.56,231.12)}},fill=black,opacity=0.799]
  (0.0000,0.0000) node[above right] (text58) {$e_{i,2}$};
\path[cm={{1.80948,0.0,0.0,1.80948,(254.88,272.747)}},fill=black,opacity=0.799]
  (0.0000,0.0000) node[above right] (text62) {$e_{i,1}$};
\path[cm={{1.80948,0.0,0.0,1.80948,(244.08,318.107)}},fill=black,opacity=0.799]
  (0.0000,0.0000) node[above right] (text66) {$e_{i,d}$};
\path[cm={{1.80948,0.0,0.0,1.80948,(374.966,181.734)}},fill=black,opacity=0.799]
  (0.0000,0.0000) node[above right] (text70) {$e_{i,d}$};
\path[cm={{1.80948,0.0,0.0,1.80948,(363.58,229.547)}},fill=black,opacity=0.799]
  (0.0000,0.0000) node[above right] (text74) {$e_{i,1}$};
\path[cm={{1.80948,0.0,0.0,1.80948,(363.23,272.747)}},fill=black,opacity=0.799]
  (0.0000,0.0000) node[above right] (text78) {$e_{i,2}$};
\path[cm={{1.80948,0.0,0.0,1.80948,(371.52,319.974)}},fill=black,opacity=0.799]
  (0.0000,0.0000) node[above right] (text82) {$e_{i,3}$};

\end{tikzpicture}
      \caption{Edges and points in the covering groupoid $E$}
      \label{fig:rules}
    \end{figure}
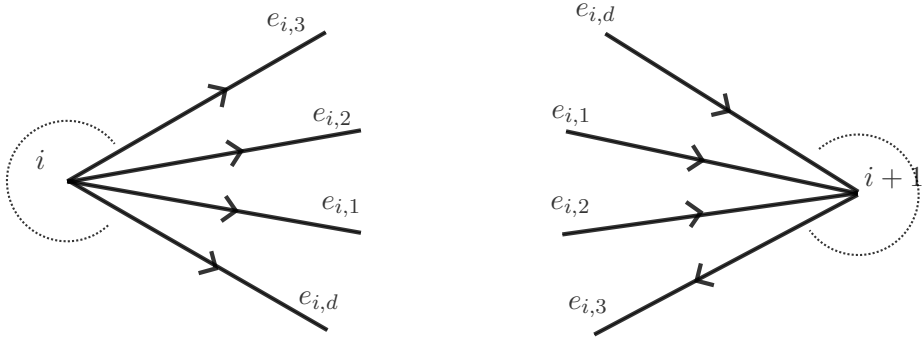

    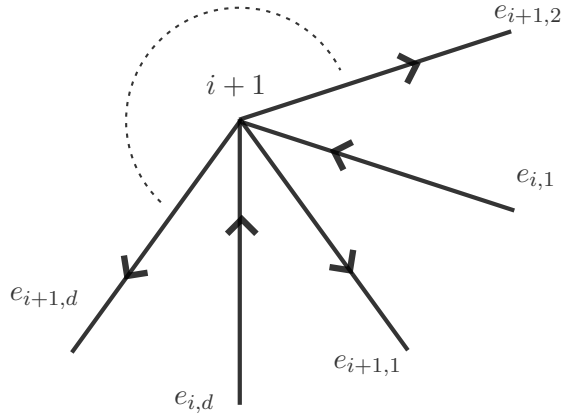
\begin{figure}[ht]
      \centering
      \begin{tikzpicture}[y=0.80pt, x=0.80pt, yscale=-1.000000, xscale=1.000000, inner sep=0pt, outer sep=0pt]
\path[draw=black,opacity=0.799,line width=1.600pt] (275.0950,263.2300) --
  (275.0950,398.4520);
\path[draw=black,opacity=0.799,line width=1.600pt] (275.1810,264.2840) --
  (195.6990,373.6800);
\path[draw=black,opacity=0.799,line width=1.600pt] (274.8180,263.5200) --
  (403.4220,221.7340);
\path[draw=black,opacity=0.799,line width=1.600pt] (276.3050,264.7150) --
  (404.9080,306.5010);
\path[draw=black,opacity=0.799,line width=1.600pt] (275.0950,263.2300) --
  (354.5760,372.6270);
\path[draw=black,opacity=0.799,line width=2.400pt] (231.5230,335.9440) --
  (221.8520,337.4750) -- (220.3200,327.8040);
\path[draw=black,opacity=0.799,line width=2.400pt] (328.7230,325.1290) --
  (327.1910,334.8000) -- (317.5200,333.2680);
\path[draw=black,opacity=0.799,line width=2.400pt] (349.9200,233.0700) --
  (358.6440,237.5160) -- (354.1990,246.2400);
\path[draw=black,opacity=0.799,line width=2.400pt] (269.1130,317.5200) --
  (276.0360,310.5960) -- (282.9600,317.5200);
\path[draw=black,opacity=0.799,line width=2.400pt] (324.1250,287.2800) ..
  controls (323.0140,285.0990) and (319.6800,278.5560) .. (319.6800,278.5560) --
  (328.4040,274.1100);
\path[cm={{1.88861,0.0,0.0,1.88861,(259.2,252.72)}},draw=black,fill=black,opacity=0.799,line
  width=0.087pt] (0.0000,0.0000) node[above right] (text32) {$i+1$};
\path[cm={{1.15999,0.0,0.0,1.15999,(166.32,353.12)}},fill=black,opacity=0.799]
  (0.0000,0.0000) node[above right] (text36) {$e_{i+1,d}$};
\path[cm={{1.15999,0.0,0.0,1.15999,(319.335,385.0)}},fill=black,opacity=0.799]
  (0.0000,0.0000) node[above right] (text40) {$e_{i+1,1}$};
\path[cm={{1.15999,0.0,0.0,1.15999,(244.08,402.28)}},fill=black,opacity=0.799]
  (0.0000,0.0000) node[above right] (text44) {$e_{i,d}$};
\path[cm={{1.15999,0.0,0.0,1.15999,(405.735,296.96)}},fill=black,opacity=0.799]
  (0.0000,0.0000) node[above right] (text48) {$e_{i,1}$};
\path[cm={{1.15999,0.0,0.0,1.15999,(395.28,219.2)}},fill=black,opacity=0.799]
  (0.0000,0.0000) node[above right] (text52) {$e_{i+1,2}$};
\path[draw=black,dash pattern=on 1.60pt off 2.40pt,opacity=0.799,line
  width=0.800pt] (237.7790,302.6030) .. controls (227.6530,292.8710) and
  (221.3510,279.1890) .. (221.3510,264.0360) .. controls (221.3510,234.4970) and
  (245.2970,210.5510) .. (274.8350,210.5510) .. controls (295.7710,210.5510) and
  (313.8970,222.5800) .. (322.6800,240.1030);

\end{tikzpicture}
      \caption{At the point $i+1$, in-edges and out-edges meet {\it alternately}}
      \label{fig:pattern}
    \end{figure}

\section{Braid group representations in automorphism groups of free groups}

  It has been shown (\cite{KimSong2018,Ghaswala2018,Salvetti2018} ) that $d$-fold ($d \geq 3$) branched covering over a disk with $n$ branch points induces an injective homomorphism from braid group $B_n$ to mapping class group $\Gamma_{g,b}$.
  In this section, we introduce an infinite family of representations of braid groups into mapping class groups.
  We generalize the approach given in \cite{KimSong2018}.
  From the $d$-fold branched cover (with $n$ branch points), we get (an infinite family of) representations from $B_n$ to $\aut F_{(d-1)(n-1)}$.
  In Theorem~\ref{thm:loops}, we give an explicit description of this representation that is the action of $\widetilde\beta_i$ on $F_{(d-1)(n-1)}$,
  the free group generated by the elements $\left\{x_{1,1},\cdots , x_{n-1,d-1}\right\}$.
  Furthermore, we give an alternative proof that $\widetilde\beta_i$ equals the product of $d-1$ Dehn twists.

  \subsection{The lifts $\widetilde{\beta}_i$}

    Let $\pi:S_{g,b}^{n}\rightarrow S_{0,1}^{n}$ be a $d$-fold branched cover with $n$ branch points.
    Let $\beta_i : S_{0,1}^{n} \rightarrow S_{0,1}^{n}$ be the half Dehn twist interchanging the two points $i$ and $i+1$ ($i = 1,2, \cdots n-1$).
    Then the lift $\widetilde{\beta_i}:S_{g,b}^{n}\rightarrow S_{g,b}^{n}$ of $\beta_i$ makes the following diagram commute :
    \[ \begin{tikzcd}
    S_{g,b}^{n} \arrow{r}{\widetilde{\beta}_i} \arrow[swap]{d}{\pi} & S_{g,b}^{n} \arrow{d}{\pi} \\%
    S_{0,1}^{n} \arrow{r}{\beta_i}& S_{0,1}^{n}
    \end{tikzcd}
    \]

    The half Dehn twist $\beta_i$ on a disk (Figure~\ref{fig:halfDehn}) may be interpreted as a self-functor of the groupoid as follows :
    $$
    \beta_i:\left\{
    \begin{array}{rcl}
      i & \mapsto & i+1 \\
      i+1 & \mapsto & i \\
      \\
      e_{i-1} & \mapsto & e_{i-1} \cdot e_{i} \\
      e_{i} & \mapsto & e_{i}^{-1} \\
      e_{i+1} & \mapsto & e_{i} \cdot e_{i+1}
    \end{array}\right.
    $$
    for each $i=1,\ldots, n-1$.
    It fixes points and edges that do not appear in the list.

    In \cite{SegalTillmann2008}, Segal and Tillmann observed that the lift of a half Dehn twist with respect to a 2-fold covering gives a full Dehn twist (Figure~\ref{fig:2fold}).
    This induces an embedding of the braid group into the mapping class group.

    \begin{figure}[ht]
      \centering
      \input{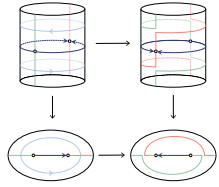}
      \caption{A half Dehn twist is lifted to a full Dehn twist}
      \label{fig:2fold}
    \end{figure}

    In \cite{KimSong2018}, a new nongeometric embedding of the braid group into the mapping class group induced by 3-fold covering was constructed.
    It was observed that the lift of a half Dehn twist is a 1/6-Dehn twist.
    We may extend this construction to the case of $d$-fold covering with $d\geq 4$.
    In this paper, we analyze the braid group representation induced by $d$-fold cover by using a new gadget, covering groupoid.

    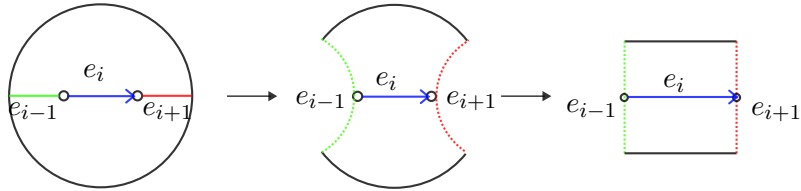
\begin{figure}[ht]
      \centering
      \definecolor{cff0200}{RGB}{255,2,0}
\definecolor{c26ff00}{RGB}{38,255,0}
\definecolor{c3cff00}{RGB}{60,255,0}
\definecolor{c000fff}{RGB}{0,15,255}

\begin{tikzpicture}[y=0.80pt, x=0.80pt, yscale=-1.000000, xscale=1.000000, inner sep=0pt, outer sep=0pt]
\path[draw=black,draw opacity=0.799,line width=0.800pt] (277.0230,228.1720) ..
  controls (269.0880,237.9790) and (256.9540,244.2500) .. (243.3560,244.2500) ..
  controls (229.3110,244.2500) and (216.8280,237.5600) .. (208.9200,227.1930);
\path[draw=black,draw opacity=0.799,line width=0.800pt] (209.2000,174.3730) ..
  controls (217.1210,164.2130) and (229.4750,157.6800) .. (243.3560,157.6800) ..
  controls (257.3630,157.6800) and (269.8170,164.3330) .. (277.7280,174.6510);
\path[draw=black,draw opacity=0.799,line width=0.800pt] (258.6190,200.9650) ..
  controls (258.6190,199.7610) and (259.5950,198.7840) .. (260.7990,198.7840) ..
  controls (262.0030,198.7840) and (262.9790,199.7610) .. (262.9790,200.9650) ..
  controls (262.9790,202.1690) and (262.0030,203.1450) .. (260.7990,203.1450) ..
  controls (259.5950,203.1450) and (258.6190,202.1690) .. (258.6190,200.9650) --
  cycle;
\path[draw=black,draw opacity=0.799,line width=0.800pt] (223.7330,200.9650) ..
  controls (223.7330,199.7610) and (224.7090,198.7840) .. (225.9140,198.7840) ..
  controls (227.1180,198.7840) and (228.0940,199.7610) .. (228.0940,200.9650) ..
  controls (228.0940,202.1690) and (227.1180,203.1450) .. (225.9140,203.1450) ..
  controls (224.7090,203.1450) and (223.7330,202.1690) .. (223.7330,200.9650) --
  cycle;
\path[draw=black,draw opacity=0.799,line width=0.800pt] (61.1387,200.7210) ..
  controls (61.1387,176.8150) and (80.5180,157.4360) .. (104.4240,157.4360) ..
  controls (128.3290,157.4360) and (147.7090,176.8150) .. (147.7090,200.7210) ..
  controls (147.7090,224.6260) and (128.3290,244.0060) .. (104.4240,244.0060) ..
  controls (80.5180,244.0060) and (61.1387,224.6260) .. (61.1387,200.7210) --
  cycle;
\path[draw=black,draw opacity=0.799,line width=0.800pt] (119.6860,200.7210) ..
  controls (119.6860,199.5170) and (120.6620,198.5400) .. (121.8660,198.5400) ..
  controls (123.0710,198.5400) and (124.0470,199.5170) .. (124.0470,200.7210) ..
  controls (124.0470,201.9250) and (123.0710,202.9010) .. (121.8660,202.9010) ..
  controls (120.6620,202.9010) and (119.6860,201.9250) .. (119.6860,200.7210) --
  cycle;
\path[draw=black,draw opacity=0.799,line width=0.800pt] (84.8007,200.7210) ..
  controls (84.8007,199.5170) and (85.7769,198.5400) .. (86.9811,198.5400) ..
  controls (88.1852,198.5400) and (89.1614,199.5170) .. (89.1614,200.7210) ..
  controls (89.1614,201.9250) and (88.1852,202.9010) .. (86.9811,202.9010) ..
  controls (85.7769,202.9010) and (84.8007,201.9250) .. (84.8007,200.7210) --
  cycle;
\path[draw=cff0200,draw opacity=0.799,line width=0.800pt] (124.0470,200.7210) --
  (147.7090,200.7210);
\path[draw=c26ff00,draw opacity=0.799,line width=0.800pt] (61.1387,200.7210) --
  (84.8007,200.7210);
\path[draw=cff0200,dash pattern=on 0.80pt,draw opacity=0.799,line width=0.800pt]
  (276.5460,226.8860) .. controls (268.4530,221.2760) and (263.1520,211.9200) ..
  (263.1520,201.3260) .. controls (263.1520,189.8750) and (269.3440,179.8710) ..
  (278.5630,174.4790);
\path[draw=c3cff00,dash pattern=on 0.80pt,draw opacity=0.799,line width=0.800pt]
  (208.7530,173.6990) .. controls (218.0640,179.0670) and (224.3320,189.1230) ..
  (224.3320,200.6430) .. controls (224.3320,212.2420) and (217.9780,222.3570) ..
  (208.5610,227.6970);
\path[draw=black,draw opacity=0.799,line width=0.800pt] (352.1380,174.9600) --
  (405.0970,174.9600);
\path[draw=c3cff00,dash pattern=on 0.80pt,draw opacity=0.799,line width=0.800pt]
  (352.1380,227.9190) -- (352.1380,174.9600);
\path[draw=black,draw opacity=0.799,line width=0.800pt] (405.0970,227.9190) --
  (352.1380,227.9190);
\path[draw=cff0200,dash pattern=on 0.80pt,draw opacity=0.799,line width=0.800pt]
  (405.0970,174.9600) -- (405.0970,227.9190);
\path[draw=black,draw opacity=0.799,line width=0.800pt] (403.4370,201.5170) ..
  controls (403.4370,200.6010) and (404.1810,199.8570) .. (405.0970,199.8570) ..
  controls (406.0140,199.8570) and (406.7570,200.6010) .. (406.7570,201.5170) ..
  controls (406.7570,202.4340) and (406.0140,203.1770) .. (405.0970,203.1770) ..
  controls (404.1810,203.1770) and (403.4370,202.4340) .. (403.4370,201.5170) --
  cycle;
\path[draw=black,draw opacity=0.799,line width=0.800pt] (350.3220,201.5170) ..
  controls (350.3220,200.6010) and (351.0650,199.8570) .. (351.9820,199.8570) ..
  controls (352.8990,199.8570) and (353.6420,200.6010) .. (353.6420,201.5170) ..
  controls (353.6420,202.4340) and (352.8990,203.1770) .. (351.9820,203.1770) ..
  controls (351.0650,203.1770) and (350.3220,202.4340) .. (350.3220,201.5170) --
  cycle;
\begin{scope}[opacity=0.799,transparency group]
  \path[draw=c000fff,opacity=0.799,line width=0.800pt] (119.6700,200.8800) ..
    controls (117.0700,200.8800) and (88.5600,200.8800) .. (88.5600,200.8800);
  \path[fill=c000fff,opacity=0.799] (116.8200,204.2670) -- (120.4870,201.2670) --
    (120.9600,200.8800) -- (120.4870,200.4930) -- (116.8200,197.4930) .. controls
    (116.6060,197.3180) and (116.2910,197.3500) .. (116.1170,197.5630) .. controls
    (115.9420,197.7770) and (115.9730,198.0920) .. (116.1870,198.2670) --
    (119.8540,201.2670) -- (119.8540,200.4930) -- (116.1870,203.4930) .. controls
    (115.9730,203.6680) and (115.9420,203.9830) .. (116.1170,204.1970) .. controls
    (116.2910,204.4100) and (116.6060,204.4420) .. (116.8200,204.2670) -- cycle;
\end{scope}
\begin{scope}[opacity=0.799,transparency group]
  \path[draw=c000fff,opacity=0.799,line width=0.800pt] (258.9900,201.1240) ..
    controls (256.3900,201.1240) and (227.8800,201.1240) .. (227.8800,201.1240);
  \path[fill=c000fff,opacity=0.799] (256.1400,204.5110) -- (259.8070,201.5110) --
    (260.2800,201.1240) -- (259.8070,200.7370) -- (256.1400,197.7370) .. controls
    (255.9260,197.5620) and (255.6110,197.5940) .. (255.4370,197.8070) .. controls
    (255.2620,198.0210) and (255.2930,198.3360) .. (255.5070,198.5110) --
    (259.1740,201.5110) -- (259.1740,200.7370) -- (255.5070,203.7370) .. controls
    (255.2930,203.9120) and (255.2620,204.2270) .. (255.4370,204.4410) .. controls
    (255.6110,204.6540) and (255.9270,204.6860) .. (256.1400,204.5110) -- cycle;
\end{scope}
\begin{scope}[opacity=0.799,transparency group]
  \path[draw=c000fff,opacity=0.799,line width=0.800pt] (404.7900,201.2770) ..
    controls (401.1760,201.2770) and (353.6050,201.2770) .. (353.6050,201.2770);
  \path[fill=c000fff,opacity=0.799] (401.9390,204.6640) -- (405.6060,201.6640) --
    (406.0790,201.2770) -- (405.6060,200.8900) -- (401.9390,197.8900) .. controls
    (401.7260,197.7150) and (401.4110,197.7470) .. (401.2360,197.9600) .. controls
    (401.0610,198.1740) and (401.0930,198.4890) .. (401.3060,198.6640) --
    (404.9730,201.6640) -- (404.9730,200.8900) -- (401.3060,203.8900) .. controls
    (401.0930,204.0650) and (401.0610,204.3800) .. (401.2360,204.5940) .. controls
    (401.4110,204.8070) and (401.7260,204.8390) .. (401.9390,204.6640) -- cycle;
\end{scope}
\path[cm={{1.0,0.0,0.0,1.0,(60.48,212.8)}},fill=black] (0.0000,0.0000)
  node[above right] (text72) {$e_{i-1}$};
\path[cm={{1.0,0.0,0.0,1.0,(96.12,193.88)}},fill=black] (0.0000,0.0000)
  node[above right] (text76) {$e_{i}$};
\path[cm={{1.0,0.0,0.0,1.0,(124.2,212.8)}},fill=black] (0.0000,0.0000)
  node[above right] (text80) {$e_{i+1}$};
\path[cm={{1.0,0.0,0.0,1.0,(196.56,207.604)}},fill=black] (0.0000,0.0000)
  node[above right] (text84) {$e_{i-1}$};
\path[cm={{1.0,0.0,0.0,1.0,(234.36,197.68)}},fill=black] (0.0000,0.0000)
  node[above right] (text88) {$e_{i}$};
\path[cm={{1.0,0.0,0.0,1.0,(267.84,207.604)}},fill=black] (0.0000,0.0000)
  node[above right] (text92) {$e_{i+1}$};
\path[cm={{1.0,0.0,0.0,1.0,(324.54,211.68)}},fill=black] (0.0000,0.0000)
  node[above right] (text96) {$e_{i-1}$};
\path[cm={{1.0,0.0,0.0,1.0,(370.44,198.72)}},fill=black] (0.0000,0.0000)
  node[above right] (text100) {$e_{i}$};
\path[cm={{1.0,0.0,0.0,1.0,(412.02,212.8)}},fill=black] (0.0000,0.0000)
  node[above right] (text104) {$e_{i+1}$};
\begin{scope}[opacity=0.799,transparency group]
  \path (293.7600,200.8800) -- (317.5200,200.8800);
    \path[draw=black,line width=0.560pt] (314.3680,200.8800) .. controls
      (312.0940,200.8800) and (293.7600,200.8800) .. (293.7600,200.8800);
    \path[fill=black] (313.3180,203.3300) -- (317.5180,200.8800) --
      (313.3180,198.4300) -- (313.3180,203.3300) -- cycle;
\end{scope}
\begin{scope}[opacity=0.799,transparency group]
  \path (164.1600,200.8800) -- (187.9200,200.8800);
    \path[draw=black,line width=0.560pt] (184.7680,200.8800) .. controls
      (182.4940,200.8800) and (164.1600,200.8800) .. (164.1600,200.8800);
    \path[fill=black] (183.7180,203.3300) -- (187.9180,200.8800) --
      (183.7180,198.4300) -- (183.7180,203.3300) -- cycle;
\end{scope}

\end{tikzpicture}
      \caption{We cut the disk along edges $e_{i-1}$ and $e_{i+1}$}
      \label{fig:cut}
    \end{figure}

    Let $S_i$ be the $d$-fold covering space over a small disk $D_i$ with two branch points $i$ and $i+1$.
    $S_i$ is called  {\it atomic surface}.
    Then $S_{g,b}^{n}$, the $d$-fold covering space of a disk with $n$ branched points, is the union of these atomic surfaces.
    Notice that each $S_i$ is homeomorphic to $S_{r,s}^{2}$, where $r = \frac{d-gcd(d,n)}{2}$ and $s = gcd(d, 2)$.

    An atomic surface $S_i$ is generated by gluing $d$ copies of $D_i$ as follows.
    We first cut out a small disk $D_i$ $(1 \leq i \leq n-1$) that contains two consecutive branch points, say $i$ and ${i+1}$.
    We cut along the edges $e_{i-1}$ and $e_{i+1}$ (Figure~\ref{fig:cut}).
    Now make $d$ copies of $D_i$.
    We glue them along $e_{i-1}$ and $e_{i+1}$ as in Figure~\ref{fig:gluing}, then we get an atomic surface $S_i$.

    \begin{figure}[ht]
      \centering
      \input{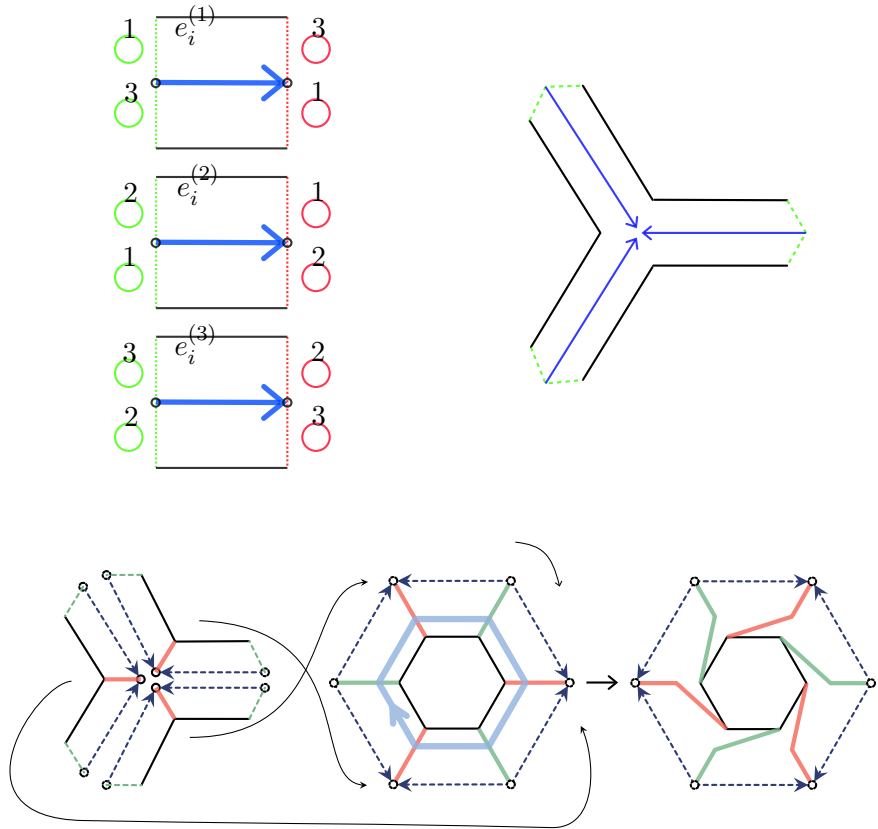}
      \caption{Gluing $d$ copies of disks to generate an atomic surface}
      \label{fig:gluing}
    \end{figure}

    The covering space $S_{g,b}^{n}$ is the union of $(n-1)$ copies of atomic surfaces $S_{r,s}^{2}$, where two neighboring $S_{r,s}^{2}$ share one branch point.
    We glue two neighboring atomic surfaces as follows.
    At the point ${i+1}$, there are $d$ incoming edges and $d$ outgoing edges.
    We glue two atomic surfaces so that the edges form the order $\left(e_{i+1,1}, e_{i,1},e_{i+1,2}, e_{i,2},\cdots e_{i+1,d}, e_{i,d}\right)$ in a counter-clockwise direction (Figure~\ref{fig:pattern}).

    \begin{figure}[ht]
      \centering
      \input{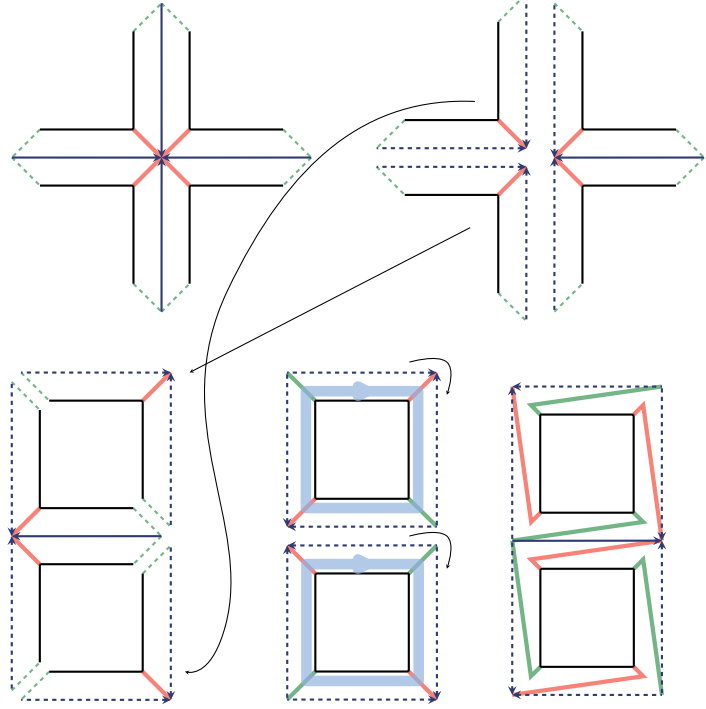}
      \caption{cut-and-paste}
      \label{fig:cut_paste}
    \end{figure}

    In the case of 4-fold cover (with $n$ branch points),
    by gluing four copies of small disks if we first get a cross shape of atomic surface as in Figure~\ref{fig:cut_paste}.
    The cross is again cut and pasted to form two squares as shown.
    The inverse image $\pi^{-1}$ of a small neighborhood around the edge $e_i$,
    is a small tubular neighborhood around the two loops (highlighted in bold).
    Now, it is easily seen that a half Dehn twist is lifted to two 1/4-Dehn twists.

    We shall look at the atomic surface $S_{r,s}^{2}$, and see how $\beta_i$ is lifted there.
    By cut-and-paste and a kind of $\frac{1}{d}$-twist (two $\frac{1}{d}$-twists for even $d$, and $\frac{1}{2d}$-twist for odd $d$), we obtain the following important result :

    \begin{theorem}
      \label{thm:lift}
      A half-Dehn twist $\beta_i$ is lifted to the covering groupoid:
      $$
      \widetilde\beta_i:\left\{
      \begin{array}{rcl}
        i & \mapsto & i+1 \\
        i+1 & \mapsto & i \\
        \\
        e_{i-1, j} & \mapsto & e_{i-1, j} \cdot e_{i, j+1} \\
        e_{i,j} & \mapsto & e_{i, j+1}^{-1} \\
        e_{i+1, j} & \mapsto & e_{i,j} \cdot e_{i+1, j}
      \end{array}\right.
      $$
    \end{theorem}

    Now we show that the homomorphism $\phi : B_n \rightarrow \Gamma_{g,b}, \beta_i \mapsto \widetilde{\beta_i}$ is well-defined, that is, the lifts $\widetilde{\beta_i}$ satisfy the braid relation.

    \begin{theorem}
      The lifts $\widetilde{\beta_i}$, as self-functors of the covering groupoid $E$, satisfy the braid relations
      $\widetilde{\beta_{i}} \widetilde{\beta}_{i+1} \widetilde{\beta_{i}} = \widetilde{\beta}_{i+1} \widetilde{\beta_{i}} \widetilde{\beta}_{i+1}$.
    \end{theorem}
    \begin{proof}
      It suffices to show
      $$\widetilde{\beta_{i}} \widetilde{\beta}_{i+1} \widetilde{\beta_{i}} (e_{k,j}) = \widetilde{\beta}_{i+1} \widetilde{\beta_{i}} \widetilde{\beta}_{i+1} (e_{k,j})$$
      for $k \in \{i-1, i, i+1, i+2\}$.
      \begin{multline*}
        e_{i-1, j} \xmapsto{\widetilde{\beta_i}} e_{i-1, j} \cdot e_{i, j+1}
        \xmapsto{\widetilde{\beta}_{i+1}} e_{i-1, j} \cdot e_{i, j+1} \cdot e_{i+1, j+2} \\
        \xmapsto{\widetilde{\beta_{i}}}  e_{i-1, j} \cdot e_{i, j+1} \cdot e_{i, j+2}^{-1} \cdot e_{i, j+2} \cdot e_{i+1, j+2} = e_{i-1, j} \cdot e_{i, j+1} \cdot e_{i+1, j+2}
      \end{multline*}

      \begin{equation*}
        e_{i, j} \xmapsto{\widetilde{\beta_i}} e_{i, j+1}^{-1}
        \xmapsto{\widetilde{\beta}_{i+1}} \left(e_{i, j+1} \cdot e_{i+1, j+2}\right)^{-1}
        \xmapsto{\widetilde{\beta_i}} \left(e_{i, j+2}^{-1} \cdot e_{i, j+2} \cdot e_{i+1, j+2}\right)^{-1}
        = e_{i+1, j+2}^{-1}
      \end{equation*}
      \begin{equation*}
        e_{i+1, j} \xmapsto{\widetilde{\beta_i}} e_{i,j} \cdot e_{i+1, j}
        \xmapsto{\widetilde{\beta}_{i+1}} e_{i,j} \cdot e_{i+1, j+1} \cdot e_{i+1, j+1}^{-1}
        \xmapsto{\widetilde{\beta_i}} e_{i, j+1}^{-1}
      \end{equation*}
      \begin{equation*}
        e_{i+2, j} \xmapsto{\widetilde{\beta_i}} e_{i+2, j}
        \xmapsto{\widetilde{\beta}_{i+1}} e_{i+1, j} \cdot e_{i+2, j}
        \xmapsto{\widetilde{\beta_i}} e_{i,j} \cdot e_{i+1, j} \cdot e_{i+2, j}
      \end{equation*}

      Whereas,
      \begin{equation*}
        e_{i-1, j} \xmapsto{\widetilde{\beta}_{i+1}}
        e_{i-1, j} \xmapsto{\widetilde{\beta_i}} e_{i-1, j} \cdot e_{i, j+1}
        \xmapsto{\widetilde{\beta}_{i+1}}
        e_{i-1, j} \cdot e_{i, j+1} \cdot e_{i+1, j+2}
      \end{equation*}
      \begin{equation*}
        e_{i, j} \xmapsto{\widetilde{\beta}_{i+1}}
        e_{i, j} \cdot e_{i+1, j+1} \xmapsto{\widetilde{\beta_i}}
        e_{i, j+1}^{-1}  \cdot e_{i, j+1} \cdot e_{i+1, j+1}
        \xmapsto{\widetilde{\beta}_{i+1}} e_{i+1, j+2}^{-1}
      \end{equation*}
      \begin{multline*}
        e_{i+1, j} \xmapsto{\widetilde{\beta}_{i+1}}
        e_{i+1, j+1}^{-1}
        \xmapsto{\widetilde{\beta_i}}
        \left(e_{i, j+1} \cdot e_{i+1, j+1}\right)^{-1}\\
        \xmapsto{\widetilde{\beta}_{i+1}}
        \left(e_{i, j+1} \cdot e_{i+1, j+1} \cdot e_{i+1, j+1}^{-1}\right)^{-1}
        = e_{i, j+1}^{-1}
      \end{multline*}
      \begin{multline*}
        e_{i+2, j}\xmapsto{\widetilde{\beta}_{i+1}}
        e_{i+1, j} \cdot e_{i+2, j} \xmapsto{\widetilde{\beta_i}}
        e_{i,j} \cdot e_{i+1, j} \cdot e_{i+2, j}\\
        \xmapsto{\widetilde{\beta}_{i+1}}
        e_{i,j} \cdot e_{i+1, j+1} \cdot \left(e_{i+1, j+1}\right)^{-1} \cdot e_{i+1, j} \cdot e_{i+2, j}
        = e_{i,j} \cdot e_{i+1, j} \cdot e_{i+2, j}
      \end{multline*}
      Thus we have $\widetilde{\beta_{i}} \widetilde{\beta}_{i+1} \widetilde{\beta_{i}} = \widetilde{\beta}_{i+1} \widetilde{\beta_{i}} \widetilde{\beta}_{i+1}$.
    \end{proof}

  \subsection{Representations of braid groups}
    In the previous section we have constructed a homomorphism $\phi_d : B_n \rightarrow \Gamma_{g,b}, \beta_i \mapsto \widetilde{\beta_i}$ induced by $d$-fold covering over a disk.
    The injectivity of this homomorphism comes from the Birman-Hilden theory (\cite{Birman1973,Birman2017,Margalit2017,Ghaswala2018}).
    On the other hand, mapping class group $\Gamma_{g,b}$($b \geq 1$) may be regarded as a subgroup of the automorphism group of the fundamental group of $S_{g,b}$ which is the free group on $2g+b-1$ generators.
    That is, the map $\phi_d$ induces a faithful representation of braid group into $\aut(F_{2g+b-1})$. In section 2.1 we have seen that $g=\frac{dn-n-d-\gcd(d,n)}{2}+1$ and $b=\gcd(d,n)$. This means that the rank of the fundamental group of the surface equals $(d-1)(n-1)$.

    Now we are going to describe this faithful representation $\phi_d : B_n \rightarrow \aut F_{(d-1)(n-1)}$ in an explicit form.
    We first determine an appropriate set of loops on the surface which are generators of the free group.
    Let $p_i$ be the path in covering groupoid $E$ defined as
    $$
    p_i := e_{0, 1} \cdot e_{1,1}\cdot \cdots \cdot e_{i-1,1}.
    $$
    The generating loops of the surface are defined to be:
    $$
    x_{i,j} := p_i \cdot \left( e_{i,j} \cdot e_{i,j+1}^{-1} \right) \cdot p_i^{-1}
    $$
    for $j = 1,2,\cdots, d-1$ and let (Figure~\ref{fig:loopxij})
    $$x_{i,d}=p_i\cdot e_{i, d}\cdot e_{i, 1}^{-1}\cdot p_i^{-1}=\left(x_{i,1}\cdot x_{i, 2}\cdot\cdots\cdot x_{i, d-1}\right)^{-1}.$$
    We read the second index $j$ modulo $d$.

    \begin{figure}[ht]
      \centering
      \input{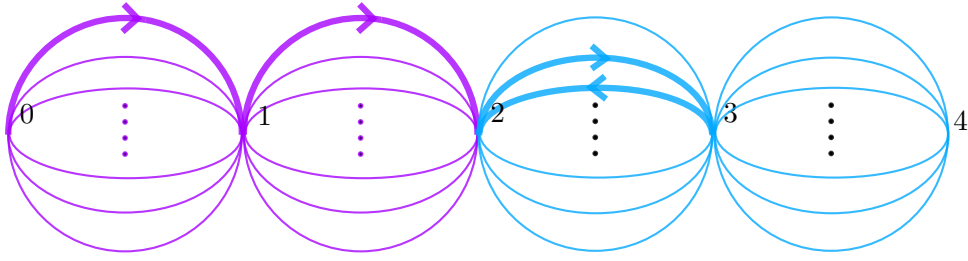}
      \caption{Loop $x_{2,2}$ is shown in bold, with three branch points}
      \label{fig:3foldloops}
    \end{figure}

    Check that the number of these loops $x_{i,j}$ is $(d-1)(n-1)$.
    Note that all these loops are nontrivial.
    For if $x_{i,j}$ is trivial then edges $e_{i,j}$ and $e_{i, j+1}$ are path homotopic in the atomic surface $S_i$,
    thus by symmetry of the atomic surface all paths $e_{i,j}$ ($j=1,\ldots,d$) are homotopic to each other,
    which is not possible.
    All these loops are the generators of the fundamental group of the surface which is a tubular neighborhood of the graph represented by the groupoid.

    \begin{figure}[ht]
      \centering
      \definecolor{c00a8ff}{RGB}{0,168,255}
\definecolor{ca700ff}{RGB}{167,0,255}

\begin{tikzpicture}[y=0.80pt, x=0.80pt, yscale=-1.000000, xscale=1.000000, inner sep=0pt, outer sep=0pt]
\path[draw=c00a8ff,draw opacity=0.799,line width=2.400pt] (498.9600,177.7820) ..
  controls (498.9600,157.4480) and (474.2350,140.9650) .. (443.7340,140.9650) ..
  controls (413.2340,140.9650) and (388.5080,157.4480) .. (388.5080,177.7820);
\path[draw=c00a8ff,draw opacity=0.799,line width=2.400pt] (439.4140,136.6940) --
  (446.2890,142.6520) -- (439.4140,147.4940);
\path[draw=c00a8ff,draw opacity=0.799,line width=2.400pt] (447.8340,161.9830) --
  (440.6400,156.4140) -- (447.2370,151.2000);
\path[draw=ca700ff,draw opacity=0.799,line width=2.400pt] (342.4240,122.2920) ..
  controls (339.2480,121.7240) and (335.9790,121.4280) .. (332.6400,121.4280) ..
  controls (302.1400,121.4280) and (277.4140,146.1530) .. (277.4140,176.6540);
\path[draw=ca700ff,draw opacity=0.799,line width=2.400pt] (387.8660,176.6540) ..
  controls (387.8660,149.4920) and (368.2580,126.9100) .. (342.4240,122.2920);
\path[draw=ca700ff,draw opacity=0.799,line width=2.400pt] (332.0760,114.9460) --
  (338.5440,121.6320) -- (331.4140,127.9090);
\path[draw=ca700ff,draw opacity=0.799,line width=2.400pt] (157.8900,121.8260) ..
  controls (154.7140,121.2580) and (151.4450,120.9620) .. (148.1060,120.9620) ..
  controls (117.6060,120.9620) and (92.8800,145.6870) .. (92.8800,176.1880);
\path[draw=ca700ff,draw opacity=0.799,line width=2.400pt] (203.3320,176.1880) ..
  controls (203.3320,149.0260) and (183.7240,126.4440) .. (157.8900,121.8260);
\path[draw=ca700ff,draw opacity=0.799,line width=2.400pt] (147.5420,114.4800) --
  (154.0100,121.1660) -- (146.8800,127.4430);
\path[draw=black,dash pattern=on 0.00pt off 3.20pt,line cap=round,draw
  opacity=0.799,line width=1.200pt] (217.0800,174.9600) -- (266.7600,174.9600);
\path[draw=c00a8ff,draw opacity=0.799,line width=2.400pt] (388.8000,177.1200) ..
  controls (388.8000,165.1910) and (413.4600,155.5200) .. (443.8800,155.5200) ..
  controls (474.3000,155.5200) and (498.9600,165.1910) .. (498.9600,177.1200);
\path[xscale=1.084,yscale=0.922,fill=black,line width=0.986pt]
  (83.2644,197.5464) node[above right] (text32) {$0$};
\path[cm={{1.192,0.0,0.0,1.192,(195.489,181.44)}},fill=black] (0.0000,0.0000)
  node[above right] (text36) {$1$};
\path[cm={{1.192,0.0,0.0,1.192,(270.744,180.863)}},fill=black] (0.0000,0.0000)
  node[above right] (text40) {$i-1$};
\path[cm={{1.192,0.0,0.0,1.192,(388.8,181.736)}},fill=black] (0.0000,0.0000)
  node[above right] (text44) {$i$};
\path[cm={{1.192,0.0,0.0,1.192,(494.64,181.44)}},fill=black] (0.0000,0.0000)
  node[above right] (text48) {$i+1$};
\path[cm={{1.0,0.0,0.0,1.0,(430.92,116.64)}},fill=black] (0.0000,0.0000)
  node[above right] (text52) {$e_{i,j}$};
\path[cm={{1.0,0.0,0.0,1.0,(416.88,173.92)}},fill=black] (0.0000,0.0000)
  node[above right] (text56) {$e_{i,j+1}^{-1}$};

\end{tikzpicture}
      \caption{The loop $x_{i,j}$}
      \label{fig:loopxij}
    \end{figure}
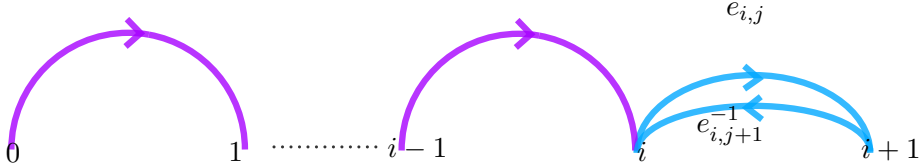

    \begin{theorem}
      \label{thm:loops}
      Let $\phi_d : B_n \rightarrow \aut F_{(d-1)(n-1)}$, $\beta_i \mapsto \widetilde{\beta_i}$ be the (faithful) representation induced by the $d$-fold covering over a disk with $n$ branch points. Then $\widetilde{\beta_{i}}$ acts on the $(d-1)(n-1)$ generators  $\{ x_{1,1}, \ldots, x_{n-1,d-1} \}$ of the free group as follows:

      $$
      \widetilde\beta_i:\left\{
      \begin{array}{rcl}
          x_{i-1, j} & \mapsto & x_{i-1, j-1}^{-1} \cdot\cdots\cdot x_{i-1, 1}^{-1} \cdot x_{i, j+1} \cdot x_{i-1, 1} \cdot\cdots\cdot x_{i-1, j}\\
          \\
          x_{i, j} & \mapsto & x_{i, 2} \cdot\cdots\cdot x_{i, j} \cdot x_{i, j+1}^{-1} \cdot\cdots\cdot x_{i, 2}^{-1}\\
          \\
          x_{i+1, j} & \mapsto & x_{i, j-1}^{-1} \cdot\cdots\cdot x_{i, 1}^{-1} \cdot x_{i+1, j} \cdot x_{i, 1} \cdot\cdots\cdot x_{i, j}
      \end{array}\right.
      $$

      where $i = 1, \ldots, n-1$, and $j = 1,\ldots, d-1$.
      Other loops $x_{k, l}$ that do not appear in the list are fixed.
    \end{theorem}
    \begin{proof}
      Note that $x_{i,1}\cdot \cdots \cdot x_{i,j} = p_i\cdot\left(e_{i,1}\cdot e_{i,j+1}^{-1}\right)\cdot p_i^{-1}$.
      \begin{align*}
        \widetilde\beta_i(x_{i-1,j}) & = p_{i-1}\cdot \widetilde\beta_i\left(e_{i-1,j}\cdot e_{i-1,j+1}^{-1}\right)\cdot p_{i-1}^{-1} \\
        & = p_{i-1}\cdot\left(e_{i-1,j}\cdot e_{i,j+1} \cdot e_{i,j+2}^{-1}\cdot e_{i-1,j+1}^{-1}\right)\cdot p_{i-1}^{-1} \\
        & = p_{i-1}\cdot\left(e_{i-1,j}\cdot e_{i-1,1}^{-1}\right)\cdot p_{i-1}^{-1} \cdot x_{i,j+1} \cdot p_{i-1} \cdot \left(e_{i-1,1}\cdot e_{i-1,j+1}^{-1}\right)\cdot p_{i-1}^{-1} \\
        & = x_{i-1, j-1}^{-1} \cdot\cdots\cdot x_{i-1, 1}^{-1} \cdot x_{i, j+1} \cdot x_{i-1, 1} \cdot\cdots\cdot x_{i-1, j}
      \end{align*}
      \begin{align*}
        \widetilde\beta_i(x_{i,j}) & = p_{i-1}\cdot \widetilde\beta_i\left(e_{i-1,1}\cdot e_{i,j}\cdot e_{i,j+1}^{-1} \cdot e_{i-1,1}^{-1}\right)\cdot p_{i-1}^{-1} \\
        & = p_{i-1}\cdot\left(e_{i-1,1}\cdot e_{i,2}\cdot e_{i,j+1}^{-1} \cdot e_{i,j+2}\cdot e_{i,2}^{-1} \cdot e_{i-1,1}^{-1}\right) \cdot p_{i-1}^{-1}\\
        & = p_{i}\cdot\left(e_{i,2}\cdot e_{i,j+1}^{-1}\right)\cdot p_{i}^{-1}\cdot p_{i}\left(e_{i,j+2}\cdot e_{i,2}^{-1}\right)\cdot p_{i}^{-1}\\
        & = x_{i, 2} \cdot\cdots\cdot x_{i, j} \cdot x_{i, j+1}^{-1} \cdot\cdots\cdot x_{i, 2}^{-1}
      \end{align*}
      \begin{align*}
        \widetilde\beta_i(x_{i+1,j}) & = p_{i-1}\cdot \widetilde\beta_i\left(e_{i-1,1}\cdot e_{i,1}\cdot e_{i+1,j}\cdot e_{i+1,j+1}^{-1}\cdot e_{i,1}^{-1}\cdot e_{i-1,1}^{-1}\right)\cdot p_{i-1}^{-1} \\
        & = p_{i-1}\cdot\left(e_{i-1,1}\cdot e_{i,j} \cdot e_{i+1,j}\cdot e_{i+1,j+1}^{-1}\cdot e_{i,j+1}^{-1} \cdot e_{i-1,1}^{-1}\right)\cdot p_{i-1}^{-1}\\
        & = p_i\cdot \left( e_{i,j}\cdot e_{i,1}^{-1} \right) \cdot p_i^{-1} \cdot x_{i+1,j} \cdot p_i \cdot \left( e_{i,1} \cdot e_{i,j+1}^{-1}\right) \cdot p_i^{-1}\\
        & = x_{i, j-1}^{-1} \cdot\cdots\cdot x_{i, 1}^{-1} \cdot x_{i+1, j} \cdot x_{i, 1} \cdot\cdots\cdot x_{i, j}
      \end{align*}
    \end{proof}

    The result of Theorem~\ref{thm:loops} may be expressed in more concise form using the conjugacy of elements in the free group.
    For this we introduce the loops $y_{i,j}$ :
    $$y_{i,j} = p_{i} \cdot e_{i,1} \cdot e_{i,j}^{-1} \cdot p_{i}^{-1} = x_{i,1} \cdot x_{i,2} \cdot\cdots\cdot x_{i, j-1}$$
    for $i = 1, \ldots, n-1$ and $j = 1, \ldots, d-1$. The loop $y_{i,j}$ is shown in Figure~\ref{fig:loopyij}.

    \begin{figure}[ht]
      \centering
      \definecolor{c00a8ff}{RGB}{0,168,255}
\definecolor{ca700ff}{RGB}{167,0,255}

\begin{tikzpicture}[y=0.80pt, x=0.80pt, yscale=-1.000000, xscale=1.000000, inner sep=0pt, outer sep=0pt]
\path[draw=c00a8ff,draw opacity=0.799,line width=2.400pt] (447.8340,161.9830) --
  (440.6400,156.4140) -- (447.2370,151.2000);
\path[draw=ca700ff,draw opacity=0.799,line width=2.400pt] (342.4240,122.2920) ..
  controls (339.2480,121.7240) and (335.9790,121.4280) .. (332.6400,121.4280) ..
  controls (302.1400,121.4280) and (277.4140,146.1530) .. (277.4140,176.6540);
\path[draw=ca700ff,draw opacity=0.799,line width=2.400pt] (387.8660,176.6540) ..
  controls (387.8660,149.4920) and (368.2580,126.9100) .. (342.4240,122.2920);
\path[draw=ca700ff,draw opacity=0.799,line width=2.400pt] (332.0760,114.9460) --
  (338.5440,121.6320) -- (331.4140,127.9090);
\path[draw=ca700ff,draw opacity=0.799,line width=2.400pt] (157.8900,121.8260) ..
  controls (154.7140,121.2580) and (151.4450,120.9620) .. (148.1060,120.9620) ..
  controls (117.6060,120.9620) and (92.8800,145.6870) .. (92.8800,176.1880);
\path[draw=ca700ff,draw opacity=0.799,line width=2.400pt] (203.3320,176.1880) ..
  controls (203.3320,149.0260) and (183.7240,126.4440) .. (157.8900,121.8260);
\path[draw=ca700ff,draw opacity=0.799,line width=2.400pt] (147.5420,114.4800) --
  (154.0100,121.1660) -- (146.8800,127.4430);
\path[draw=black,dash pattern=on 0.00pt off 3.20pt,line cap=round,draw
  opacity=0.799,line width=1.200pt] (217.0800,174.9600) -- (266.7600,174.9600);
\path[draw=c00a8ff,draw opacity=0.799,line width=2.400pt] (388.8000,177.1200) ..
  controls (388.8000,165.1910) and (413.4600,155.5200) .. (443.8800,155.5200) ..
  controls (474.3000,155.5200) and (498.9600,165.1910) .. (498.9600,177.1200);
\path[draw=ca700ff,draw opacity=0.799,line width=2.400pt] (453.5180,121.8260) ..
  controls (450.3420,121.2580) and (447.0730,120.9620) .. (443.7340,120.9620) ..
  controls (413.2340,120.9620) and (388.5080,145.6870) .. (388.5080,176.1880);
\path[draw=ca700ff,draw opacity=0.799,line width=2.400pt] (498.9600,176.1880) ..
  controls (498.9600,149.0260) and (479.3510,126.4440) .. (453.5180,121.8260);
\path[draw=ca700ff,draw opacity=0.799,line width=2.400pt] (443.1700,114.4800) --
  (449.6380,121.1660) -- (442.5080,127.4430);
\path[cm={{1.0,0.0,0.0,1.0,(430.92,92.88)}},fill=black] (0.0000,0.0000)
  node[above right] (text34) {$e_{i,1}$};
\path[cm={{1.0,0.0,0.0,1.0,(423.36,174.96)}},fill=black] (0.0000,0.0000)
  node[above right] (text38) {$e_{i,j}^{-1}$};
\path[cm={{1.4649,0.0,0.0,1.4649,(87.8591,183.6)}},fill=black] (0.0000,0.0000)
  node[above right] (text42) {$0$};
\path[cm={{1.192,0.0,0.0,1.192,(195.489,181.44)}},fill=black] (0.0000,0.0000)
  node[above right] (text46) {$1$};
\path[cm={{1.192,0.0,0.0,1.192,(270.744,180.863)}},fill=black] (0.0000,0.0000)
  node[above right] (text50) {$i-1$};
\path[cm={{1.192,0.0,0.0,1.192,(388.8,181.736)}},fill=black] (0.0000,0.0000)
  node[above right] (text54) {$i$};
\path[cm={{1.192,0.0,0.0,1.192,(494.64,181.44)}},fill=black] (0.0000,0.0000)
  node[above right] (text58) {$i+1$};

\end{tikzpicture}
      \caption{The loop $y_{i,j}$}
      \label{fig:loopyij}
    \end{figure}
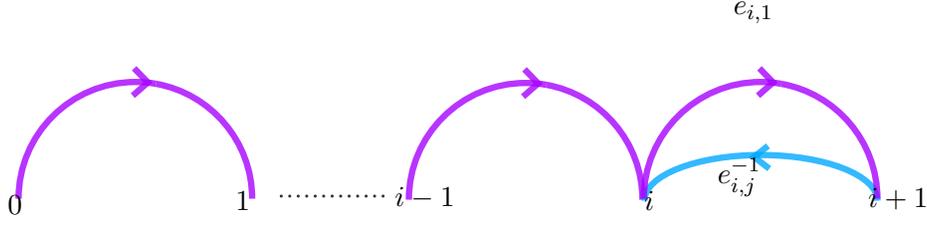

    \begin{corollary}
      \label{cor:yij}
      Let $\phi_d : B_n \rightarrow \aut F_{(d-1)(n-1)}$, $\beta_i \mapsto \widetilde{\beta_i}$ be the (faithful) representation induced by the $d$-fold covering over a disk with $n$ branch points.
      Then $\widetilde{\beta_{i}}$ acts on the generators $\{ x_{1,1}, \ldots, x_{n-1,d-1} \}$ of the free group as follows:

      $$
      \widetilde\beta_i:\left\{
      \begin{array}{rcl}
        x_{i-1, j} & \mapsto & y_{i-1,j}^{-1} \cdot x_{i, j+1} \cdot y_{i-1, j+1}\\
        \\
        x_{i, j} & \mapsto & x_{i,1}^{-1} \cdot y_{i,j+1} \cdot y_{i, j+2}^{-1} \cdot x_{i,1}\\
        \\
        x_{i+1, j} & \mapsto & y_{i,j}^{-1} \cdot x_{i+1,j} \cdot y_{i, j+1}

      \end{array}\right.
      $$
      Equivalently, we have
      $$
      \widetilde\beta_i:\left\{
      \begin{array}{rcl}
        x_{i-1, j} & \mapsto &  (x_{i, j+1})^{y_{i-1,j}} \cdot x_{i-1,j}\\
        \\
        x_{i, j} & \mapsto & (x_{i,j+1}^{-1})^{y_{i,j+1}^{-1} \cdot\, x_{i,1}}\\
        \\
        x_{i+1, j} & \mapsto & (x_{i+1,j})^{y_{i,j}} \cdot x_{i,j}

      \end{array}\right.
      $$
      Here, $x^y$ means the conjugate $y^{-1} x y$.
    \end{corollary}

    The lift $\widetilde{\beta_i}$, regarded as an element in a mapping class group,
    may be expressed in terms of Dehn twists which generate the mapping class group.
    Let $D_{x_{i,j}}$ denote the Dehn twist along the loop $x_{i,j}$.
    Then the action of this Dehn twist on the covering groupoid is as follows :
    \begin{lemma}
      \label{lem:dehn}

      $$
      D_{x_{i,j}}:\left\{
      \begin{array}{rcl}
        e_{i-1, j} & \mapsto & e_{i-1, j} \cdot e_{i, j+1}, \\
        e_{i-1, k} & \mapsto & e_{i-1, k} \cdot e_{i,j} \qquad\textrm{ if } k\neq j, \\
        \\
        e_{i,j} & \mapsto & e_{i, j+1}^{-1}, \\
        e_{i, j+1} & \mapsto & e_{i,j}^{-1}, \\
        e_{i, k} & \mapsto & e_{i,j}^{-1} \cdot e_{i, k} \cdot e_{i,j}^{-1} \qquad \textrm{ if } k\neq j, j+1, \\
        \\
        e_{i+1, j+1} & \mapsto & e_{i, j+1} \cdot e_{i+1, j+1}, \\
        e_{i+1, k} & \mapsto & e_{i,j} \cdot e_{i+1, k} \qquad \textrm{ if } k\neq j+1.
      \end{array}\right.
      $$
    \end{lemma}

    From Theorem~\ref{thm:lift} and Lemma~\ref{lem:dehn}, by hand calculation (or computer program),
    we can see that $\widetilde\beta_i$ equals the product of $d-1$ Dehn twists (not the inverse of Dehn twists as in \cite{KimSong2018}) :
    \begin{theorem}
      \label{thm:Dehn}
      The lift $\widetilde\beta_i$ equals $D_{x_{i, 2}}\cdot D_{x_{i,3}}\cdot \cdots \cdot D_{x_{i,d}}$ as an element of mapping class group.
    \end{theorem}

    This theorem implies that the embedding $\phi_d:B_n\rightarrow \Gamma_{g,b}$ induced by $d$-fold covering is nongeometric for each $d\geq 3$.

\section*{Acknowledgments}
  The third author was supported by the Korean National Research Fund NRF-2016R1D1A1B03934531.

\end{document}